\documentclass{amsart}
\usepackage{amssymb}
\usepackage{amscd}
\usepackage{a4wide}
\usepackage[utf8]{inputenc}

 \newcommand{\bdet}[1]{\: \mathrm{det}{(#1)}}
  \def\hsi{\hat{\sigma}}
   \def\bsi{\bar{\sigma}}
\newcommand{\topform}{\overline{\omega}}
\newcommand{\ol}{\overline{\lambda}}
  \def\rend#1#2{{{\rm End}\sb{#1}(#2)}}
  \def\1{\mathbb{I}}
\newcommand{\NN}{\mathbb{N}}
\newcommand{\ZZ}{\mathbb{Z}}
\newcommand{\Hom}[3]{\mathrm{Hom}_{#1}(#2,#3)}

  \newtheorem{proposition}{Proposition}[section]
  \newtheorem{lemma}[proposition]{Lemma}
  \newtheorem{corollary}[proposition]{Corollary}
  \newtheorem{theorem}[proposition]{Theorem}

  \theoremstyle{definition}
  \newtheorem{definition}[proposition]{Definition}

  \theoremstyle{remark}
  \newtheorem{remark}[proposition]{Remark}

\begin{document}
\markboth{S.Karacuha and C.Lomp}{Integral calculus on quantum exterior
algebras.}
\title{Integral calculus on quantum exterior algebras}
\author{SERKAN KARA\c{C}UHA}
\address{Department of Mathematics, FCUP, University of Porto, Rua Campo Alegre
687, 4169-007 Porto, Portugal}
\author{CHRISTIAN LOMP}
\address{Department of Mathematics, FCUP, University of Porto, Rua Campo Alegre
687, 4169-007 Porto, Portugal}
\thanks{The authors would like to thank the referee for his/her valuable remarks as well as Tomasz Brzezi\'nski for his comments on the first draft of this paper. This research was funded by the European Regional Development Fund through the programme COMPETE and by the Portuguese Government through the FCT
- Fundação para a Ciência e a Tecnologia under the project PEst-C/MAT/UI0144/2011. The first author was supported by the grant SFRH/BD/51171/2010.}

\begin{abstract}
Hom-connections and associated integral forms have been introduced
and studied by T.Brzezi\'nski as an adjoint version of the usual
notion of a connection in non-commutative geometry. Given a flat
hom-connection on a differential calculus $(\Omega, d)$ over an
algebra $A$ yields the integral complex which for various algebras
has been shown to be isomorphic to the noncommutative de Rham complex
(in the sense of Brzezi\'nski et al. \cite{BrzezinskiKaoutitLomp}).
In this paper we shed further light on the question when the integral and the de Rham complex are isomorphic for an algebra $A$ with a
flat hom-connection.
We specialise our study to the case where an $n$-dimensional differential
calculus can be constructed on a quantum exterior algebra over an $A$-bimodule. Criteria are given for free bimodules with
diagonal or upper triangular bimodule structure. Our results are
illustrated for a differential calculus on a multivariate quantum
polynomial algebra and for a differential calculus on Manin's
quantum $n$-space.
\end{abstract}

\maketitle

\section{Introduction}
Let $A$ be an algebra over a field $K$. A derivation
$d:A\rightarrow \Omega^1$ of a $K$-algebra $A$ into an
$A$-bimodule is a $K$-linear map satisfying the Leibniz rule
$d(ab)=ad(b)+d(a)b$ for all $a,b \in A$. The pair $(\Omega^1,d)$
is called a \emph{first order differential calculus} (FODC) on
$A$. More generally a differential graded algebra
$\Omega=\bigoplus_{n\geq 0} \Omega^n$ is an $\NN$-graded algebra
with a linear mapping $d:\Omega\rightarrow \Omega$ of degree $1$
that satisfies $d^2=0$ and the graded Leibinz rule. This means
that $d(\Omega^n)\subseteq \Omega^{n+1}$, $d^2=0$ and for all
homogeneous elements $a,b \in \Omega$ the graded Leibniz rule:
\begin{equation}
d(ab)= d(a)b + (-1)^{|a|}ad(b)
\end{equation}
holds, where $|a|$ denotes the degree of $a$, i.e. $a\in
\Omega^{|a|}$ (see for example \cite{Connes}). We shall call
$(\Omega, d)$ an $n$-dimensional differential calculus on $A$ if
$\Omega^m=0$ for all $m\geq n$. The zero component $A=\Omega^0$ is
a subring of $\Omega$ and hence $\Omega^n$ are $A$-bimodule for
all $n>0$. In particular $d:A\rightarrow \Omega^1$ is a bimodule
derivation and $(\Omega^1,d)$ is an FODC over $A$. The elements of
$\Omega^n$ are then called $n$-forms and the product of $\Omega$
is denoted by $\wedge$. Given an FODC $(\Omega^1,d)$ over $A$, a
connection in a right $A$-module $M$ is a $K$-linear map
$\nabla^0: M \rightarrow M \otimes_A\Omega^1 $ satisfying
\begin{equation} \nabla^0(ma) = \nabla^0(m)a + m\otimes_A d(a) \qquad \forall a\in A,
m\in M. \end{equation} In \cite{Brzezinski2008} T.Brzezinski
introduced an adjoint version of a connection by defining the
notion of a right hom-connection as a pair $(M,\nabla_0)$, where
$M$ is a right $A$-module and $\nabla_0:
\Hom{A}{\Omega^1}{M}\rightarrow M$ is a $K$-linear map such that
\begin{equation} \nabla_0(fa) = \nabla_0(f)a + f(d(a)) \qquad \forall a\in A, f\in
\Hom{A}{\Omega^1}{M}\end{equation} Here the multiplication
$(fa)(\omega):= f(a\omega)$, for all $\omega \in\Omega^1$, makes
$\Hom{A}{\Omega^1}{M}$ a right $A$-module. In case the FODC stems
from a differential calculus $(\Omega, d)$, then a hom-connection
$\nabla_0$ on $M$ can be extended to maps
$\nabla_{m}:\Hom{A}{\Omega^{m+1}}{M} \longrightarrow
\Hom{A}{\Omega^{m}}{M}$ with
\begin{equation} \nabla_{m}(f)(v) = \nabla(fv) + (-1)^{m+1} f(dv), \qquad
\forall f\in \Hom{A}{\Omega^{m+1}}{M}, v\in
\Omega^m.\end{equation}
If $\nabla_0\nabla_1 = 0$, the
hom-connection $\nabla_0$ is called flat. In this paper we will be mostly interested in the case $M=A$. Set $\Omega_m^\ast :=
\Hom{A}{\Omega^m}{A}$ as well as $\Omega^\ast = \bigoplus_m
\Omega_m^\ast$ and define
 $\nabla:\Omega^\ast \rightarrow \Omega^\ast$ by $\nabla(f) = \nabla_m(f)$ for
all $f\in \Omega_{m+1}^\ast$.

If $\nabla_0$ is flat, then $(\Omega^\ast, \nabla)$ builds up the {\it integral
complex}:
\[\begin{CD}\label{diagram1}
\cdots @>\nabla_{3}>> \Omega_3^\ast  @>\nabla_{2}>> \Omega_2^\ast @>\nabla_{1}>>
\Omega_1^\ast  @>\nabla_0>>  A
\end{CD}\]
It had been shown in \cite{Brzezinski2010,BrzezinskiKaoutitLomp}
that for some finite dimensional differential calculi the integral
complex is isomorphic to the {\it de Rham complex} given by
$(\Omega, d)$:
\[\begin{CD}\label{diagram2}
A @>d>> \Omega_1  @>d>> \Omega_2 @>d>> \Omega_3 @>d>>\cdots
\end{CD}\]
i.e. for certain algebras $A$ and $n$-dimensional differential
calculi $\Omega = \bigoplus_{m=0}^n \Omega^m$ it had been proven
that there is a commutative diagram

\[\begin{CD}\label{diagram3}
\Omega_n^\ast @>\nabla_{n-1}>> \Omega_{n-1}^\ast  @>\nabla_{n-2}>> {\cdots}
@>\nabla_{1}>> \Omega_1^\ast @>\nabla_0>> A\\
@A{\Theta_0}AA        @A{\Theta_1}AA           @.   @A{\Theta_{n-1}}AA
@A{\Theta}AA\\
A               @>d>> \Omega^1         @>d>> {\cdots} @>d>> \Omega^{n-1} @>d>>
\Omega^n
\end{CD}\]
in which vertical maps are right $A$-module isomorphisms: In this
case, we say that $A$ satisfies the {\it strong Poincar\'e
duality} with respect to $(\Omega, d)$ and $\nabla$, following
T.Brzezinski \cite{Brzezinski2010}.
 The purpose of this paper is to provide further examples of algebras
and differential calculi whose corresponding de Rham and integral complexes are isomorphic,  which contribute to the general study of algebras with this property. The reader should be warned that the Poincaré duality in the sense of M.Van den Bergh \cite{vandenbergh} (see also the  work of U.Krähmer \cite{Kraehmer}) is different.

\section{Twisted multi-derivations and hom-connections}

From Woronowicz' paper \cite{Woronowicz} it follows that any
covariant differential calculus on a quantum group is determined by
a certain family of maps which had been termed {\it twisted
multi-derivations} in \cite{BrzezinskiKaoutitLomp}. We recall from \cite{BrzezinskiKaoutitLomp} that by a {\em right twisted
multi-derivation} in an algebra $A$ we mean a pair $(\partial, \sigma)$, where
$\sigma: A\to M_n(A)$ is an algebra homomorphism ($M_n(A)$ is the algebra of
$n\times n$ matrices with entries from $A$) and $\partial :A\to A^n$ is a
$k$-linear map such that, for all $a\in A$, $b\in B$,
\begin{equation}\label{eq.partial}
\partial(ab) = \partial(a)\sigma(b) + a\partial(b).
\end{equation}
Here $A^n$ is understood as an $(A$-$M_n(A))$-bimodule. We write
$\sigma(a) = (\sigma_{ij}(a))_{i,j=1}^n$ and $\partial(a) =
(\partial_i(a))_{i=1}^n$ for an element $a\in A$. Then
\eqref{eq.partial} is equivalent to the following $n$ equations
\begin{equation}
\partial_i(ab) = \sum_j \partial_j(a)\sigma_{ji}(b) + a\partial_i(b), \qquad
i=1,2,\ldots, n.
\end{equation}
Given a right twisted multi-derivation $(\partial, \sigma)$ on $A$ we construct
an FODC on the free left $A$-module
\begin{equation}\Omega^1= A^n=\bigoplus_{i=1}^n A\omega_i\end{equation} with
basis $\omega_1,\ldots, \omega_n$ which becomes an $A$-bimodule by
$\omega_i  a = \sum_{j=1}^n\sigma_{ij}(a)\omega_j$ for all $1\leq i \leq n$. The
map \begin{equation}d: A \to \Omega^1, \qquad  a \mapsto \sum_{i=1}^n
\partial_i(a) \omega_i\end{equation}
is a derivation and makes $(\Omega^1, d)$ a first order differential calculus on
$A$.

A map $\sigma : A \to M_n(A)$ can be equivalently understood as an element of
$M_n(\rend k A)$. Write $\bullet$ for the product in $M_n(\rend k A)$,  $\1$ for
the unit in $M_n(\rend k A)$ and $\sigma^T$ for the transpose of $\sigma$.
\begin{definition}\label{def.der.free}
Let $(\partial, \sigma)$ be a right twisted multi-derivation. We say that
$(\partial, \sigma)$ is {\em free}, provided there exist algebra maps $\bsi: A
\to M_n(A)$ and $\hsi: A \to M_n(A)$ such that
\begin{equation}\label{bar.sigma3}
\bsi \bullet \sigma^T = \1, \qquad \sigma^T\bullet\bsi = \1\, ,
\end{equation}
\begin{equation}\label{hat.sigma}
\hsi \bullet \bsi^T = \1, \qquad \bsi^T\bullet\hsi = \1\, .
\end{equation}
\end{definition}

Theorem \cite [Theorem 3.4]{BrzezinskiKaoutitLomp} showed that for
any free right twisted multi-derivation $(\partial, \sigma; \bsi
,\hsi)$ on $A$, and associated first order differential calculus
$(\Omega^1,d)$ with generators $\omega_i$, the map
\begin{equation}\label{def.hom.twist}
\nabla : \Hom{A}{\Omega^1}{A}\to A, \qquad f\mapsto \sum_i \partial_i^\sigma
\left( f\left(\omega_i\right)\right)\, .
\end{equation}
is a hom-conection, where $\partial_i^\sigma := \sum_{j,\,k}
\bsi_{kj}\circ\partial_j\circ \hsi_{ki}$, for each $i=1,2,\ldots, n$.
Moreover $\nabla$ had been shown to be unique with respect to the property that
$\nabla (\xi _i) =0$, for all $i = 1,2, \ldots , n$, where $\xi_i : \Omega^1\to
A$ are right $A$-linear maps defined by $\xi_i(\omega_j) = \delta_{ij}$, $i,j =
1,2,\ldots , n$.

We shall be mostly interested in right twisted multi-derivation
$(\partial,\sigma)$ that are upper triangular, for which
$\sigma_{ij}=0$ for all $i>j$ holds. It had been shown in
\cite[Proposition 3.3]{BrzezinskiKaoutitLomp} that an upper
triangular right twisted multi-derivation is free if and only if
$\sigma_{11}, \ldots, \sigma_{nn}$ are automorphisms of $A$.


\section{Differential calculi on quantum exterior algebras}
Let $A$ be a unital associative algebra over a field $K$. Given an
$A$-bimodule $M$ which is free as left and right $A$-module with
basis $\{ \omega_1, \ldots, \omega_n\}$ one defines the tensor
algebra of $M$ over $A$ as
\begin{equation}
T_A(M) = A \oplus M \oplus (M\otimes M) \oplus M^{\otimes 3} \oplus \cdots =
\bigoplus_{n=0}^\infty M^{\otimes n}
\end{equation}
which is a graded algebra whose product is the concatenation of tensors and
whose zero component is $A$.
Following \cite[I.2.1]{BrownGoodearl} we call an $n\times n$-matrix $Q=(q_{ij})$
over $K$ a
{\it multiplicatively antisymmetric matrix} if
$q_{ij}q_{ji}=q_{ii}=1$ for all $i,j$.  The {\it quantum exterior algebra} of
$M$ over $A$ with respect to a
multiplicatively antisymmetric matrix $Q$ is defined as
$$ \bigwedge\!\!{}^Q (M) := T_A(M)/\langle \omega_i\otimes \omega_j +
q_{ij}\omega_j\otimes \omega_i, \omega_i\otimes \omega_i \mid
i,j=1,\ldots, n\rangle.$$ This construction for a vector space
$M=V$ and a field $A=K$ appears in \cite{NaiduShroffWhitherspoon,
Ueyama}. The product of $\bigwedge^{Q}(M)$ is written as $\wedge$.
The quantum exterior algebra is a free left and right $A$-module
of rank $2^n$ with basis
$$ \{1\}\cup \{ \omega_{i_1} \wedge \omega_{i_2} \cdots \wedge \omega_{i_k} \mid
i_1 < i_2 < \cdots < i_k, \, 1\leq k\leq n\}.$$
Write $\mathrm{sup}(\omega_{i_1} \wedge \omega_{i_2} \cdots \wedge \omega_{i_k})
= \{i_1, i_2 \cdots, i_k\}$ for any basis element.
Given a bimodule derivation $d:A\rightarrow M$, we will examine when $d$ can be
extended to an exterior derivation of $\bigwedge^Q(M)$, i.e. to a graded map
$d:\bigwedge^Q(M)\rightarrow \bigwedge^Q(M)$ of degree $1$ such that $d^2=0$ and
such that the graded Leibniz rule is satisfied.


\begin{proposition}\label{ExtendingDiffStructToSkewExteriorAlgberas}
Let $(\partial, \sigma)$ be a right twisted multi-derivation of
rank $n$ on a $K$-algebra $A$ with associated FODC $(\Omega^1,d)$.
Let $Q$ be an $n\times n$ multiplicatively antisymmetric matrix over $k$. Then
$d:A\rightarrow \Omega^1$ can be extended to make
$\Omega=\bigwedge^Q(\Omega^1)$ an $n$-dimensional differential
calculus on $A$ with $d(\omega_i)=0$ for all $i=1, \ldots, n$ if
and only if
\begin{equation}\label{qcomm}
\partial_i\partial_j = q_{ji}\partial_j\partial_i,
\qquad \mbox{ and } \qquad
\partial_i\sigma_{kj}-q_{ji}\partial_j\sigma_{ki} =
q_{ji}\sigma_{kj}\partial_i-\sigma_{ki}\partial_j \qquad \forall i<j, \, \forall
k.
\end{equation}
\end{proposition}

\begin{proof}
Suppose $d$ extends to make $\Omega$ a differential calculus on $A$ with
$d(\omega_i)=0$.
Then for all $a\in A$ and $k=1,\ldots,n$ the following equations hold:
\begin{equation}
d(\omega_k a) = d(\omega_k)a - \omega_k \wedge d(a)
= \sum_{j=1}^n -\omega_k\wedge \partial_j(a)\omega_j =
\sum_{i,j=1}^n -\sigma_{ki}(\partial_j(a)) \omega_i\wedge \omega_j
\end{equation}
\begin{equation}
d\left(\sum_{j=1}^n\sigma_{kj}(a)\omega_j\right)
= \sum_{i,j=1}^n \partial_i(\sigma_{kj}(a)) \omega_i \wedge \omega_j +
\sum_{j=1}^n\sigma_{kj}(a)d(\omega_j)
= \sum_{i,j=1}^n \partial_i(\sigma_{kj}(a)) \omega_i \wedge \omega_j
\end{equation}
Hence, as $\omega_ka = \sum_{j=1}^n\sigma_{kj}(a)\omega_j$ and
$\omega_j\wedge \omega_i = -q_{ji} \omega_i\wedge \omega_j$ for
$i<j$, we have
\begin{equation}
-\sigma_{ki}\partial_j + q_{ji} \sigma_{kj}\partial_i = \partial_i\sigma_{kj} -
q_{ji}\partial_j\sigma_{ki}
\qquad \forall i<j
\end{equation}
Furthermore $d^2=0$ implies for all $a\in A$:
\begin{equation}
0 = d^2(a) = \sum_{i,j=1}^n \partial_i\partial_j(a) \omega_i \wedge \omega_j =
\sum_{i<j}
(\partial_i\partial_j-q_{ji}\partial_j\partial_i)(a) \omega_i \wedge \omega_j,
\end{equation}
which shows $\partial_i\partial_j = q_{ji}\partial_j\partial_i$, for $i<j$.

On the other hand if (\ref{qcomm}) holds, then set for any homogeneous element
$a\omega\in\Omega^m$ with $a\in A$ and $\omega=\omega_{j_1}\wedge
\omega_{j_2}\wedge\cdots\wedge\omega_{j_m}$, with $j_1<j_2 <\cdots < j_m$, a
basis element of $\Omega^m$:
\begin{equation}
d(a\omega) := d(a)\wedge \omega = \sum_{i=1}^n \partial_i(a)\omega_i \wedge
\omega_{j_1}\wedge \omega_{j_2}\wedge\cdots\wedge\omega_{j_m}.
\end{equation}
We will show that $d:\Omega\rightarrow \Omega$ in that way, will satisfy $d^2=0$
and the graded Leibniz rule.
For any $a\omega\in \Omega^m$ as above:
\begin{equation}
d^2(a\omega) = \sum_{i,j=1}^n \partial_i\partial_j(a) \omega_i \wedge \omega_j
\wedge \omega =
 \sum_{i<j}^n (\partial_i\partial_j-q_{ji}\partial_j\partial_i)(a) \omega_i
\wedge \omega_j \wedge \omega = 0
\end{equation}
Since (\ref{twisted_relation}) implies that $\partial_i(1)=\sum_{j}\partial_j(1)\sigma_{ji}(1)+\partial_i(1) = 2\partial_i(1)$, as  $\sigma_{ji}(1)=0$ if $i\neq j$, we have $\partial_i(1) = 0$ and hence$d(\omega_i)=d(1)\wedge \omega_i = 0$ for all $i$.

We prove the graded Leibniz rule
\begin{equation}\label{gradedLeibniz}
d(a\omega \wedge b\nu)=d(a\omega)\wedge b\nu + (-1)^m a\omega \wedge d(b\nu)
\end{equation}
inductively on the grade of $\omega$, where
$\omega=\omega_{j_1}\wedge \cdots \wedge \omega_{j_m}$ and
$\nu=\omega_{i_1}\wedge \cdots \wedge \omega_{i_k}$ are basis
elements of $\Omega$ and $a,b\in A$. For a $m=0$, ie. $a\omega =
a$, equation (\ref{gradedLeibniz}) follows from the definition and
$d(\nu)=0$. Let $m>0$ and suppose that (\ref{gradedLeibniz}) has
been proven for all basis elements $\omega$ of grade $|\omega|\leq
m-1$. Let $\omega$ be a basis element with $|\omega|=m$ and write
$\omega=\omega'\wedge \omega_k$.
\begin{eqnarray*}
d(a\omega\wedge b\nu) &=& d(a\omega'\wedge \omega_k \wedge b\nu) \\
&=& \sum_{j=1}^n d\left(a\omega'\wedge \sigma_{kj}(b) \omega_j \wedge \nu\right)
\\\label{induction_hypoth}
&=& \sum_{j=1}^n d(a\omega')\wedge \sigma_{kj}(b) \omega_j \wedge \nu +
(-1)^{m-1} \sum_{j=1}^n a\omega' \wedge
d\left(\sigma_{kj}(b) \omega_j \wedge \nu\right) \\
&=& d(a\omega')\wedge \omega_k \wedge b\nu + (-1)^{m-1} a\omega' \wedge
\sum_{i,j=1}^n \partial_i(\sigma_{kj}(b)) \omega_i\wedge  \omega_j \wedge \nu
\\\label{eq_gr_Leibniz_1}
&=& d(a\omega) \wedge b\nu - (-1)^{m} a\omega' \wedge  \sum_{i<j} \left[
\partial_i(\sigma_{kj}(b)) - q_{ji}
\partial_j(\sigma_{ki}(b))\right] \omega_i\wedge  \omega_j \wedge \nu
\\\label{eq_gr_Leibniz_2}
&=& d(a\omega) \wedge b\nu + (-1)^{m} a\omega' \wedge  \sum_{i<j} \left[
\sigma_{ki}(\partial_j(b)) - q_{ji}
\sigma_{kj}(\partial_i(b))\right] \omega_i\wedge  \omega_j \wedge \nu \\
&=& d(a\omega) \wedge b\nu + (-1)^{m} a\omega' \wedge  \sum_{i,j=1}^n
\sigma_{ki}(\partial_j(b)) \omega_i\wedge  \omega_j \wedge \nu \\
&=& d(a\omega) \wedge b\nu + (-1)^{m} a\omega' \wedge  \omega_k \wedge
\sum_{j=1}^n \partial_j(b)\omega_j \wedge \nu \\
&=& d(a\omega)\wedge b\nu + (-1)^{m} a\omega \wedge d(b\nu)
\end{eqnarray*}
which shows the graded Leibniz rule, where the induction hypothesis has been
used in the third line and where (\ref{qcomm}) has been used in the sixth line .
\end{proof}

Suppose that $(\partial, \sigma)$ is a free right twisted multi-derivation satisfying the equations $(\ref{qcomm})$ and that $(\Omega,d)$
is the associated $n$-dimensional differential calculus over $A$ for some
$n\times n$ matrix $Q$. Then, as mentioned above, $\nabla: \Hom{A}{\Omega^1}{A}
\rightarrow A$ with
$\nabla(f) = \sum_{i=1}^n \partial_i^\sigma(f(\omega_i))$ for all $f\in
\Hom{A}{\Omega^1}{A}$ is hom-connection.
For each $1\leq m < n$ one defines also
$\nabla_{m}:\Hom{A}{\Omega^{m+1}}{A} \longrightarrow \Hom{A}{\Omega^{m}}{A}$
with
\begin{equation} \nabla_{m}(f)(u) = \nabla(fu) +
(-1)^{m+1} f(d(u)), \qquad \forall f\in \Hom{A}{\Omega^{m+1}}{A},
u\in \Omega^m,\end{equation} where $fu\in \Hom{A}{\Omega^{1}}{A}$
is defined by $fu(v)=f(u\wedge v)$ for all $v \in \Omega^1$. As
every element $u\in \Omega^m$ can be uniquely written as a right
$A$-linear combination of basis elements
$\omega=\omega_{i_1}\wedge \cdots \wedge \omega_{i_m}$ and since
$\nabla_m(f)$ is right $A$-linear and furthermore by Proposition
\ref{ExtendingDiffStructToSkewExteriorAlgberas} $d(\omega)=0$ is
satisfied, we conclude that for $u=\omega a$:
\begin{equation} \nabla_{m}(f)(\omega a) = \nabla_{m}(f)(\omega)a =
\nabla(f\omega)a + (-1)^{m+1} f(d(\omega))a = \nabla(f\omega)a\end{equation}
holds. If $\partial_i^\sigma(1)=0$ for all $i$, the hom-connection is flat,
because for any dual basis element $f=\beta_{s,t} \in \Hom{A}{\Omega^2}{A}$ with
$s<t$, i.e. $\beta_{s,t}(\omega_i\wedge \omega_j) = \delta_{s,i}\delta_{t,j}$
one has
$$
\nabla(\nabla_{1}(f)) = \sum_{i=1}^n {\partial_i^\sigma}(\nabla_1(f)(\omega_i))
= \sum_{i=1}^n {\partial_i^\sigma}(\nabla(f\omega_i)) = \sum_{i=1}^n\sum_{j=1}^n
{\partial_i^\sigma}({\partial_j^\sigma}(f(\omega_i\wedge \omega_j)))=
{\partial_s^\sigma}( {\partial_t^\sigma}(1)) = 0.
$$

Set $\Omega^* = \Hom{A}{\Omega}{A} = \bigoplus_{m=0}^n
\Hom{A}{\Omega^m}{A}$ and note that $\nabla$ induces a map of
degree $-1$ on $\Omega^*$. We want to establish an isomorphism
between the de Rham complex given by $d:\Omega \rightarrow \Omega$
and the integral complex given by $\nabla: \Omega^* \rightarrow
\Omega^*$. More precisely we are looking for a bijective chain map
$\Theta: (\Omega, d) \rightarrow (\Omega^*, \nabla)$ such that the
following diagram commutes:

\[\begin{CD}\label{diagram}
A               @>d>> \Omega^1         @>d>> {\cdots} @>d>> \Omega^{n-1} @>d>>
\Omega^n\\
@V{\Theta_0}VV        @V{\Theta_1}VV           @.   @V{\Theta_{n-1}}VV
@V{\Theta_n}VV\\
\Hom{A}{\Omega^n}{A} @>>\nabla_{n-1}> \Hom{A}{\Omega^{n-1}}{A}
@>>\nabla_{n-2}> {\cdots} @>>\nabla_{1}> \Hom{A}{\Omega^1}{A} @>>\nabla > A
\end{CD}\]

One attempt is to define the maps $\Theta_m$ via the dual basis element of
$\Omega^n$.
Define $$\topform = \omega_1 \wedge \cdots \wedge \omega_n \in \Omega^n$$ for
the base element of $\Omega^n$.
Let $\beta \in {\Omega^n}^*$ be the dual basis of $\Omega^n$ as a right
$A$-module, i.e. $ \beta(\topform a) = a$ for all $a\in A$.
For any $0\leq m < n$ define $\Theta_m: \Omega^m \longrightarrow
\Hom{A}{\Omega^{n-m}}{A}$ through
$\Theta_m(v) =  (-1)^{m(n-1)}\beta v$ for all $v\in \Omega^m$. Note that
$\Theta_n=\beta$. Moreover the maps $\Theta_m$ are right $A$-linear taking into
account the right $A$-module structure of $\Hom{A}{\Omega^{n-m}}{A}$, namely for
$a\in A, v\in \Omega^m$ and $w\in \Omega^{n-m}$:
$$\Theta_m(va)(w) = (-1)^{m(n-1)}\beta(va\wedge w) = (-1)^{m(n-1)}\beta(v \wedge
aw) = \Theta_m(v)(aw) = (\Theta_m(v)a)(w).$$
Hence $\Theta_m(va)=\Theta_m(v)a$.

For a certain class of twisted multi-derivations, extended to a
quantum exterior algebra, we will show that the maps $\Theta_m$
are always isomorphisms. We say that a twisted multi-derivation
$(\partial, \sigma)$ on an algebra $A$ is upper triangular if
$\sigma_{ij}=0$ for all $i>j$. By \cite[Proposition
3.3]{BrzezinskiKaoutitLomp} any upper triangular twisted
multi-derivation is free if and only if $\sigma_{ii}$ are
automorphisms of $A$ for all $i$. The corresponding maps $\bsi$
and $\hsi$ are defined inductively by $\bsi_{ii}=\sigma_{ii}^{-1}$
for all $i$,
$\bsi_{ij}=-\sum_{k=j}^{i-1}\sigma_{ii}^{-1}\sigma_{ki}\bsi_{kj}$
for all $i>j$ and $\bsi_{ij}=0$ for $i<j$. The map $\hsi$ is
defined analogously using $\bsi$.

\begin{theorem}\label{uppertriangular}
 Let $(\partial, \sigma)$ be a free upper triangular twisted multi-derivation on
$A$ with associated FODC $(\Omega^1,d)$.
 Suppose that $d:A\rightarrow \Omega^1$ can be extended to an $n$-dimensional
differential calculus
 $(\Omega, d)$ where $\Omega=\bigwedge^Q (\Omega^1)$ is the quantum exterior
algebra of $\Omega^1$ for some matrix $Q$.
 Then the following hold:
 \begin{enumerate}
  \item $\topform a = \bdet{\sigma}(a) \: \topform$, for all $a \in A$, where
$\bdet{\sigma}=\sigma_{11} \circ \cdots \circ \sigma_{nn}$.
  \item The maps $\Theta_m: \Omega^m \rightarrow \Hom{A}{\Omega^{n-m}}{A}$ given
by $\Theta_m(v)= (-1)^{m(n-1)}\beta v$ for all $v\in\Omega^m$ are
isomorphisms of right $A$-modules.
\item Moreover if
\begin{equation}\label{poincare-cond}
\partial_i^\sigma = \left(\prod_j q_{ij}\right) \bdet{\sigma}^{-1} \partial_i
\bdet{\sigma} \qquad \forall i=1,\ldots, n
\end{equation}
holds, then $\Theta = (\Theta_m)_{m=0}^n$ is a chain map, that is,
$A$ satisfies the strong Poincar\'e duality with respect to
$(\Omega, d)$ in the sense of T.Brzezinski.
 \end{enumerate}
\end{theorem}

\begin{proof}
(1)  By the definition of the bimodule structure of $\bigwedge^Q(\Omega^1)$
and by the fact that $\bsi$ is lower triangular we have
$$a\topform = \sum_{j_n\geq n} \cdots \sum_{j_1\geq 1} \omega_{j_1}\wedge \cdots
\wedge \omega_{j_n} \bsi_{n j_n}\circ \cdots \circ \bsi_{1 j_1}(a).$$
By the definition of the quantum exterior algebra the non-zero terms
$\omega_{j_1}\wedge \cdots \wedge \omega_{j_n}$ must have distinct indices,
i.e. $j_k \neq j_l$ for all $k\neq l$. In particular $j_n=n$ and hence
inductively we can conclude that $j_i=i$ for all $i$. This shows that $a\topform
= \topform \bdet{\sigma}^{-1}(a)$.

(2) For every basis element of $\omega = \omega_{i_1}\wedge \cdots \wedge
\omega_{i_{n-m}}$ of $\Omega^{n-m}$,
there exists a unique complement basis element $\omega' = \omega'_{j_1} \wedge
\cdots \wedge \omega'_{j_{m}}$ of $\Omega^{m}$
such that $\omega' \wedge \omega  \neq 0$. Let $C_\omega$ be the non-zero scalar
such that $\omega'\wedge \omega = C_{\omega} \topform$
Let $f\in\Hom{A}{\Omega^{n-m}}{A}$ be any non-zero element and set
$$a_\omega = (-1)^{m(n-1)}C_{\omega}^{-1}\bdet{\sigma}(f(\omega))$$ for any
basis element $\omega\in\Omega^{n-m}$.
Set $v= \sum a_{\omega} \omega'$. Then
$$ \Theta_m(v)(\omega) =(-1)^{m(n-1)} \beta(a_\omega \omega'\wedge \omega)
=(-1)^{m(n-1)} \beta( {a_\omega C_\omega \topform}) =
\bdet{\sigma}^{-1}(\bdet{\sigma}(f(\omega)) = f(\omega).$$ Hence
$\Theta_m(v)=f$, which shows that $\Theta_m$ is surjective. To
prove injectivity, assume that $v=\sum a_{\omega} \omega' \in
\Omega^m$ is an element such that $\Theta_m(v)$ is the zero
function. Then for any basis element $\omega \in \Omega^{n-m}$,
one has
$$\Theta_m(v)(\omega) = (-1)^{m(n-1)} \beta( a_{\omega} \omega' \wedge \omega)
= (-1)^{m(n-1)}C_\omega \bdet{\sigma}^{-1}(a_\omega) = 0$$
which implies $a_\omega$ to be zero. Thus $v=0$ and $\Theta_m$ is an
isomorphism.

(3) We will show that $(\Theta_m)_m$ is a chain map, i.e. that
$\Theta_{m+1}\circ d = \nabla_{n-m-1}\circ \Theta_m$.
Let $\omega=\omega_{j_1}\wedge \cdots\wedge \omega_{j_m}$ be a basis element of
$\Omega^m$ and let $a\in A$.
For any basis element $\nu=\omega_{k_1}\wedge \cdots \wedge
\omega_{k_{n-m-1}}\in \Omega^{n-m-1}$ we have
$$ \Theta_{m+1}(d(a\omega))(\nu) = (-1)^{(m+1)(n-1)}\sum_{i=1}^n \beta(
\partial_i(a)\omega_i\wedge \omega \wedge \nu).$$
On the other hand
$$\nabla_{n-m-1}(\Theta_m(a\omega))(\nu) = (-1)^{m(n-1)}
\nabla(\beta(a\omega\wedge \nu) = (-1)^{m(n-1)}\sum_{i=1}^n
\partial_i(\beta(a\omega\wedge \nu \wedge \omega_i)),$$ as
$d(\nu)=0$. Note that $\Theta_{m+1}(d(a\omega))(\nu)=0$ and
$\nabla_{n-m-1}(\Theta_m(a\omega))(\nu)=0$ if
$\mathrm{sup}(\omega)\cap \mathrm{sup}(\nu) \neq \emptyset$. Hence
suppose that $\omega$ and $\nu$ have disjoint support. Then there
exists a unique index $i$ that does not belong to
$\mathrm{sup}(\omega)\cup \mathrm{sup}(\nu)$. Let $C$ be the
constant such that
$$ \omega\wedge \nu \wedge \omega_i = C \topform.$$
Recall also that by the definition of the quantum exterior algebra we have:
$$\omega_i \wedge \omega\wedge \nu = \left( \prod_{j\neq i} -q_{ij}\right )
\omega \wedge \nu  \wedge \omega_i = (-1)^{n-1}
C\left( \prod_j q_{ij}\right ) \topform.$$
Note that hypothesis (\ref{poincare-cond}) is moreover equivalent to
\begin{equation}
 \partial_i^\sigma \circ \bdet{\sigma}^{-1}  = \left(\prod_j
q_{ij}\right) \bdet{\sigma}^{-1} \circ\partial_i
\end{equation}
These equations yield now the following:
\begin{eqnarray*}
\Theta_{m+1}(d(a\omega))(\nu) &=&
(-1)^{(m+1)(n-1)} \beta( \partial_i(a)\omega_i\wedge \omega\wedge \nu)\\
&=& (-1)^{m(n-1)} C \left(\prod_j q_{ij}\right)  \beta( \partial_i(a)
\topform)\\
&=& (-1)^{m(n-1)} C \left(\prod_j q_{ij}\right)
\bdet{\sigma}^{-1}(\partial_i(a)) \\
&=& (-1)^{m(n-1)} C\: \partial_i^\sigma\left(\bdet{\sigma}^{-1} (a)\right)\\
&=& \partial_i^\sigma \left( (-1)^{m(n-1)} C\: \beta(a \topform) \right)\\
&=& \partial_i^\sigma \left((-1)^{m(n-1)} \beta(a\omega\wedge\nu\wedge
\omega_i)\right)
=  \nabla_{n-m-1}\left(\Theta_m(a\omega)(\nu)\right)
\end{eqnarray*}
Thus $\Theta_{m+1}\circ d = \nabla_{n-m-1}\circ \Theta_m$. Hence
$\Theta$ is a chain map between the de Rham and the integral
complexes of right $A$-modules.

\end{proof}

\begin{remark}
 Let  $(\partial, \sigma)$  be an upper-triangular twisted multi-derivation of
rank $n$ on $A$ and let $Q$ be an $n\times n$ matrix with
 $q_{ij}q_{ji}=q_{ii}=1$. The conditions to extend the multi-derivations to the
quantum exterior algebra $\Omega=\bigwedge^Q(\Omega^1)$
 such that the complex of integral forms on $A$ and the de Rham complex are isomorphic with respect to $(\Omega, d)$ are:
 \begin{enumerate}
  \item $\sigma_{ii}$ is an automorphism of $A$ for all $i$;
  \item $\partial_i\partial_j = q_{ji}\partial_j\partial_i$ for all $i<j$;
  \item $\partial_i\sigma_{kj} - q_{ji}\sigma_{kj}\partial_i =
q_{ji}\partial_j\sigma_{ki} -\sigma_{ki}\partial_j$ for all
  $i<j$ and all $k$;
  \item $\partial_i^\sigma = \left(\prod_j q_{ij}\right) \bdet{\sigma}^{-1}
\partial_i  \bdet{\sigma}$ for all $i$.
 \end{enumerate}
\end{remark}

\section{Differential calculi from skew derivations}

The simplest bimodule structure on $\Omega^1 = A^n$ is a {\it
diagonal} one, i.e. if  $\sigma_{ij}= \delta_{ij} \sigma_i$ for
all $i,j$ where $\sigma_1, \ldots, \sigma_n$ are endomorphisms of
$A$. Moreover if $\sigma$ is diagonal and $(\partial,\sigma)$ is a
right twisted multi-derivation  on $A$, then the maps $\partial_i$
are right $\sigma_i$-derivations, i.e. for all $a,b\in A$ and $i$:
\begin{equation}
\partial_i(ab)=\partial_i(a)\sigma_i(b) + a\partial_i(b).
\end{equation}
Conversely, given any right $\sigma_i$-derivations $\partial_i$ on
$A$, for $i=1,\ldots, n$ one can form a corresponding diagonal
twisted multi-derivation $(\partial,\sigma)$ on $A$. Such {\it
diagonal} twisted multi-derivation $(\partial,\sigma)$ is free if
and only if the maps $\sigma_1, \ldots, \sigma_n$ are
automorphisms. The associated $A$-bimodule structure on
$\Omega^1=A^n$ with left $A$-basis $\omega_1, \ldots, \omega_n$ is
given by $\omega_i a = \sigma_i(a)\omega_i$ for all $i$ and $a\in
A$. From Proposition
\ref{ExtendingDiffStructToSkewExteriorAlgberas} we obtain the
following corollary for diagonal bimodule structures.

\begin{corollary}\label{extend_diff_structure_cor}\label{
generalIsoDeRhamIntegral}  Let $A$ be an algebra over a field $K$,
$\sigma_i$ automorphisms and $\partial_i$  right $\sigma_i$-skew derivations on
$A$, for $i=1,\ldots, n$ and let $(\Omega^1,d)$ be the associated first order
differential calculus on $A$.
\begin{enumerate}
 \item The derivation $d:A\rightarrow \Omega^1$ extends to an $n$-dimensional
differential calculus ($\Omega, d)$ where
 $\Omega=\bigwedge^Q(\Omega^1)$ is the quantum exterior algebra with respect to
some $Q$  such that $d(\omega_i)=0$ for all $i=1,\ldots, n$ if and only if
\begin{equation}\label{qcomm_diag}
\partial_i\sigma_j = q_{ji}\sigma_j\partial_i\qquad  \mbox{ and }
\qquad \partial_i\partial_j = q_{ji}\partial_j\partial_i \qquad \forall
i<j\end{equation}
\item If $\partial_i\sigma_j = q_{ji}\sigma_j\partial_i$ for all $i,j$ and
$\partial_i\partial_j = q_{ji}\partial_j\partial_i$ for all $i<j$,
then the de Rham and the integral complexes on $A$ are isomorphic
relative to $(\Omega, d)$.
\end{enumerate}
\end{corollary}
\begin{proof}
(1) Since $\sigma_{ki}=0$ for all $k\neq i$, equation (\ref{qcomm}) reduces to
equation (\ref{qcomm_diag}).

(2) Note that $\partial_i^\sigma = \sigma_i^{-1}\partial_i\sigma_i =
\partial_i$.
On the other hand by hypothesis
$\partial_i \bdet{\sigma} = \left(\prod_j q_{ji}\right)
\bdet{\sigma}\partial_i$.
Hence
$$\left(\prod_j q_{ij} \right) \bdet{\sigma}^{-1}\partial_i \bdet{\sigma} =
\partial_i = \partial_i^\sigma.$$
Thus by Theorem \ref{uppertriangular}, $A$ satisfies the strong
Poincaré duality with respect to $(\Omega, d)$ in the sense of T.Brzezinski.
\end{proof}

\section{Multivariate quantum polynomials}
Let $K$ be a field, $n>1$, and $Q=(q_{ij})$ a $n\times n$ multiplicatively
antisymmetric matrix over $K$.
The multivariate quantum polynomial algebra with respect to $Q$ is defined as:
$$ A=\mathcal{O}_Q(K^n) := K\langle x_1, \ldots, x_n \rangle/\langle
x_ix_j-q_{ij}x_jx_i \: | \: 1\leq i,j\leq n\rangle.$$ This means
that $x_i$ and $x_j$ commute up to the scalar $q_{ij}$ in $A$.
Moreover every element is a linear combination of ordered
monomials $x^\alpha = x_1^{\alpha_1}\cdots x_n^{\alpha_n}$ with
$\alpha=(\alpha_1, \ldots, \alpha_n) \in \mathbb{N}^n$.  The set
of $n$-tuples $\NN^n$ is a submonoid of $\ZZ^n$ by componentwise
addition. For any $\alpha\in \ZZ$ we set $x^\alpha = 0$ if there
exists $i=1,\ldots, n$ such that $\alpha_i<0$. Furthermore $\NN^n$
is partially ordered as follows:  $\alpha\leq \beta$ if and only
if $\alpha_i\leq \beta_i$,$i=1,\cdots ,n$ for $\alpha,\beta \in
\NN^n$. If $\alpha\leq \beta$, then $\beta-\alpha \in \NN^n$ and
$x^{\beta-\alpha}\neq 0$.

For two generic monomials $x^\alpha$ and $x^\beta$ with $\alpha, \beta \in
\NN^n$ one has
\begin{equation} x^\alpha x^\beta = \left( \prod_{1\leq j<i\leq n}
q_{ij}^{\alpha_i \beta_j}\right) x^{\alpha+\beta} =
\mu(\alpha,\beta) x^{\alpha+\beta},\end{equation} where
$\mu(\alpha,\beta)=\prod_{1\leq j<i\leq n} q_{ij}^{\alpha_i
\beta_j}$. The algebra $A$ has been well-studied by Artamonov
\cite{Artamonov, ArtamonovWisbauer} as well by Goodearl and Brown
\cite{BrownGoodearl} and others. The Manin's quantum n-space is
obtained in case there exists $q\in K$ with $q_{ij}=q$ for all
$i<j$. In particular for $n=2$ one obtains the quantum plane.

We define automorphisms $\sigma_1, \ldots, \sigma_n$ and right
$\sigma_i$-derivations of $A$ as follows: For a generic monomial
$x^\alpha$ with $\alpha \in \mathbb{N}^n$ one sets
\begin{equation}
\sigma_i(x^\alpha) := \lambda_i(\alpha) x^\alpha
\qquad \mathrm{and} \qquad \partial_i(x^\alpha) := \alpha_i \delta_i(\alpha)
x^{\alpha-\epsilon^i}\end{equation}
where $\lambda_i(\alpha) = \prod_{j=1}^n q_{ij}^{\alpha_j}$,  $\delta_i(\alpha)
= \prod_{i<j} q_{ij}^{\alpha_j}$ and $\epsilon^i \in \mathbb{N}^n$ such that
$\epsilon^i_j=\delta_{ij}$.
Let $\overline{\delta_i}(\alpha)=\prod_{i>j} q_{ij}^{\alpha_j}$ and note that
$\lambda_i(\alpha)=\delta_i(\alpha)\overline{\delta_i}(\alpha)$.
Since $\mu(\alpha,\beta)=\mu(\alpha-\epsilon^i,\beta)\overline{\delta_i}(\beta)$
if $\alpha_i\neq 0$ and
$\mu(\alpha,\beta)=\mu(\alpha,\beta-\epsilon^i)\delta_i(\alpha)^{-1}$ if
$\beta_i\neq 0$, we have:
\begin{eqnarray*}
\partial_i(x^\alpha x^\beta)
&=&
(\alpha_i+\beta_i)\mu(\alpha,\beta)\delta_i(\alpha+\beta)x^{
\alpha+\beta-\epsilon^i}\\
&=&
\alpha_i\mu(\alpha-\epsilon^i,\beta)\overline{\delta_i}
(\beta)\delta_i(\alpha)\delta_i(\beta)x^{\alpha-\epsilon^i+\beta}
+
\beta_i\mu(\alpha,\beta-\epsilon^i)\delta_i(\alpha)^{-1}
\delta_i(\alpha)\delta_i(\beta)x^{\alpha+\beta-\epsilon^i}\\
&=& \alpha_i\delta_i(\alpha)x^{\alpha-\epsilon^i} \lambda_i(\beta)x^{\beta} +
x^\alpha \beta_i \delta_i(\beta)x^{\beta-\epsilon^i}\\
&=& \partial_i(x^\alpha)\sigma_i(x^\beta) + x^\alpha \partial_i(x^\beta)
\end{eqnarray*}

Let $i<j$ and $\alpha \in \NN^n$. Then
$\delta_j(\alpha-\epsilon^i)=\delta_j(\alpha)$, while
$\delta_i(\alpha-\epsilon^j)=\delta_i(\alpha)q_{ji}$. Hence
\begin{equation}
\partial_j(\partial_i(x^\alpha)) = \alpha_i \alpha_j
\delta_i(\alpha)\delta_j(\alpha-\epsilon^i) x^{\alpha-\epsilon^i-\epsilon^j}
= \alpha_i \alpha_j  q_{ij}\delta_i(\alpha-\epsilon^j)\delta_j(\alpha)
x^{\alpha-\epsilon^i-\epsilon^j}
= q_{ij}\partial_i(\partial_j(x^\alpha)) \end{equation}
Thus $\partial_j\partial_i = q_{ij}\partial_i\partial_j$ for all $i<j$.

Let $i\leq j$ and $\alpha \in \NN^n$. Then
\begin{equation}
\sigma_i(\partial_j(x^\alpha))=\alpha_j\delta_j(\alpha)\lambda_i(\alpha-\epsilon^j)x^{
\alpha-\epsilon^j}=
\alpha_j\delta_j(\alpha)\lambda_i(\alpha)q_{ji}x^{\alpha-\epsilon^j}=q_{ji}
\lambda_i(\alpha)\partial_j(x^{\alpha})=q_{ji}
\partial_j(\sigma_i(x^\alpha)).
\end{equation}
Hence $\sigma_i\partial_j = q_{ji}\partial_j\sigma_i$ for all $i\leq j$.
By Corollary \ref{extend_diff_structure_cor} we can conclude:

\begin{corollary}
Let $A=\mathcal{O}_Q(K^n)$ be the multivariate quantum polynomial algebra and
let $\Omega=\bigwedge^Q(\Omega^1)$ be the associated quantum exterior algebra.
Then the derivation $d:A\rightarrow \Omega^1$ with $d(x^\alpha)=\sum_{i=1}^n
\partial_i(x^\alpha) \omega_i$ makes $\Omega$ into a differential calculus such
that the de Rham complex and the integral complex are isomorphic.
\end{corollary}

\section{Manin's quantum $n$-space}
In this section we will show that for a special case of the multivariate quantum
polynomial algebra there exists a differential calculus whose bimodule structure
is not diagonal, but upper triangular and nevertheless the de Rham complex and
the integral complex are isomorphic.

Let $q\in K\setminus\{0\}$.
For the matrix $Q=(q_{ij})$ with $q_{ij}=q$ and $q_{ji}=q^{-1}$ for all $i<j$
and
$q_{ii}=1$, the algebra $\mathcal{O}_Q(K^n)$ is called the
{\it coordinate ring of quantum $n$-space} or {\it Manin's quantum $n$-space}
and will be denoted by
$A=K_q[x_1,\ldots, x_n]$.
We have the following defining relations of the algebra $A$
\begin{equation} x_{i}x_{j}=qx_{j}x_{i}, \qquad i<j.\end{equation}

Note that for  $\alpha \in \NN^n$ and $1\leq i \leq n$ we have:
$$\lambda_i(\alpha) x^\alpha x_i \: = \: x^{\alpha+\epsilon^i} \: = \: \ol_i(\alpha)
x_ix^\alpha,$$
where
$$\lambda_i(\alpha)=\prod_{i<j} q^{\alpha_j} \qquad \mbox{ and }\qquad
\ol_i(\alpha)=\prod_{j<i} q^{-\alpha_j}.$$
More generally
$$ x^{\alpha + \beta} = \left(\prod_{j=1}^{n-1} \lambda_j(\alpha)^{\beta_j}\right)
x^\alpha x^\beta = \prod_{1\leq s<j\leq n}q^{\alpha_s\beta_j}
x^\alpha x^\beta$$ Let $\mu(\alpha,\beta)$ be the scalar such that
$x^\alpha x^\beta=\mu(\alpha,\beta)x^{\alpha+\beta}$.

We take the following two-parameter first order differential
calculus $\Omega^{1}$ (see \cite[p.468]{KlimykSchmudgen} for the case $p=q^2$ and \cite[Example 3.9]{BrzezinskiKaoutitLomp} for the case $n=2$), which is freely
generated by $\{\omega_1, \ldots \omega_n\}$ over $A$ subject to
the relations
\begin{equation}\label{FODC-rel1} \omega_ix_{j}=qx_{j}\omega_i +
(p-1)x_{i}\omega_{j}, \qquad i<j,\end{equation}
\begin{equation}\label{FODC-rel2}
\omega_{i}x_{i}=px_{i}\omega_{i},\end{equation}
\begin{equation} \omega_{j}x_{i}=pq^{-1}x_{i}\omega_{j}, \qquad
i<j,\end{equation} There exists an algebra map $\sigma:
A\rightarrow M_n(A)$ whose associated matrix of endomorphisms
$\sigma=(\sigma_{ij})$ is upper triangular and such that $\omega_i
x^\alpha = \sum_{i\leq j} \sigma_{ij}(x^\alpha) \omega_j.$ The
next lemma will characterize the algebra map $\sigma$. For any
$\alpha\in\NN^n$ and $i=1,\ldots,n $ set $\pi_i(\alpha) =
\prod_{s<i} p^{\alpha_s}$.

\begin{lemma}\label{sigma-Manin-space}
For $\alpha\in\NN^n$ the entries of the matrix
$\sigma(x^{\alpha})$ are as follows $\sigma_{ij}(x^\alpha)=0$ for
$i>j$ and
\begin{equation*}\label{sigma_relations}
\sigma_{ij}(x^\alpha) = \eta_{ij}(\alpha)x^{\alpha+\epsilon^{i}-\epsilon^{j}}
\qquad \mbox { where }\qquad \eta_{ij}(\alpha) = \left\{\begin{array}{lcl}
\pi_j(\alpha)\ol_i(\alpha)\lambda_j(\alpha)(p^{\alpha_j}-1) & \mbox{for} &
i<j,\\[3mm]
\pi_i(\alpha)\ol_i(\alpha)\lambda_i(\alpha) p^{\alpha_i} & \mbox{for} & i=j
\end{array}\right.
\end{equation*}
\end{lemma}

\begin{proof}
Fix a number $i$ between $1$ and $n$.
We prove the relations for $\sigma_{ij}$ by induction on the length of $\alpha$,
which by length we mean $ |\alpha| = \alpha_1 + \cdots  + \alpha_n$.
For $|\alpha|=0$ the relation is clear, because $\alpha_j=0$ for all $j$, i.e.
$x^\alpha = 1$. Hence $\omega_i x^\alpha = \omega_i$, i.e.
$\sigma_{ij}(x^{\alpha})=\delta_{ij}$. Since $p^{\alpha_j}-1=0$ for all $j$ and
$p^{\alpha_i}=1$ the relation holds.

Now suppose that $m\geq 0$ and that the relations (\ref{sigma_relations}) hold
for all $\alpha \in \NN^n$ of length $m$. Let $\beta\in \NN^n$ be an element of
length $m+1$ and let $k$ be the largest index $j$ such that $\beta_j\neq 0$. Set
$\alpha=\beta - \epsilon^k$, i.e. $\beta=\alpha + \epsilon^k$.
We have to discuss the three cases $k<i$, $k=i$ and $k>i$.

If $k<i$, then for all $i<j$, $\alpha_j=0$, i.e. $\sigma_{ij}(x^\alpha)=0$.
Hence
\begin{equation*}
\omega_i x^\beta = \omega_i x^\alpha x_k = \sigma_{ii}(x^\alpha)\omega_ix_k =
pq^{-1}\sigma_{ii}(x^\alpha) x_k \omega_i
= p\pi_i(\alpha) q^{-1}\ol_i(\alpha)x^{\alpha}x_k \omega_i=
\pi_i(\beta)\ol_i(\beta)x^{\beta}\omega_i,
\end{equation*}
since $\lambda_i(\alpha)=p^{\alpha_i}=1$, $\pi_i(\alpha +
\epsilon^k)=p\pi_i(\alpha)$ and
$\ol_i(\alpha + \epsilon^k)=q^{-1}\ol_i(\alpha)$ for any $k<i$ and
$\alpha\in\NN^n$.
Thus
$\sigma_{ii}(x^\beta)=\pi_i(\beta)\ol_i(\beta)\lambda_i(\beta)p^{\beta_i}
x^\beta$.

If $k=i$, then again $\sigma_{ij}(x^\alpha)=0$ for all $j>i$. Moreover
$\lambda_j(\alpha)=1$ for all $j>i$. Thus
\begin{equation*}
\omega_i x^\beta = \sigma_{ii}(x^\alpha)\omega_ix_i = \sigma_{ii}(x^\alpha) p
x_i \omega_i
= \pi_i(\alpha) \ol_i(\alpha) p^{\alpha_i+1}x^{\alpha}x_i\omega_i
= \pi_i(\beta) \ol_i(\beta) p^{\beta_i}x^{\beta}\omega_i,
\end{equation*}
since $\alpha_s=\beta_s$ for all $s<i$, i.e. $\pi_i(\beta)=\pi_i(\alpha)$ and
$\ol_i(\beta)=\ol_i(\alpha)$.

If $i<k$, then note that $\sigma_{ij}(x^\alpha)=0$ for all $k<j$, because
$p^{\alpha_j}=1$. Thus
\begin{eqnarray*}
\omega_i x^\beta &=& \sigma_{ii}(x^\alpha)\omega_i x_k +
\sum_{i<j<k} \sigma_{ij}(x^\alpha) \omega_j x_ k + \sigma_{ik}(x^\alpha)
\omega_k x_k \\
&=& \sigma_{ii}(x^\alpha)[qx_k\omega_i + (p-1)x_i\omega_k]+
\sum_{i<j<k} \sigma_{ij}(x^\alpha) [qx_k\omega_j + (p-1)x_j\omega_k] +
\sigma_{ik}(x^\alpha) p x_k\omega_k \\
&=& q\sigma_{ii}(x^\alpha)x_k \omega_i
+ \sum_{i<j<k} q\sigma_{ij}(x^\alpha)x_k\omega_j
+ \underbrace{\left[ (p-1)\sigma_{ii}(x^\alpha)x_i +
\sum_{i<j<k} (p-1)\sigma_{ij}(x^\alpha)x_j +
p\sigma_{ik}(x^\alpha)x_k\right]}_{(*)}\omega_k
\end{eqnarray*}
Note that for any $j<k$ we have $q\lambda_j(\alpha)=\lambda_j(\beta)$. Hence
$q\sigma_{ij}(x^\alpha)x_k = \sigma_{ij}(x^\beta)$ for all $j<k$.
It is left to show that the expression $(*)$ equals $\sigma_{ik}(x^\beta)$.
Recall that $\lambda_l(\alpha)x^\alpha x_l = x^{\alpha+\epsilon^l}$. Hence
$\lambda_j(\alpha)x^{\alpha+\epsilon^{i}-\epsilon^{j}}x_j =
x^{\alpha+\epsilon^i}$.
Note also that $p^{\alpha_j}\pi_j(\alpha)=\pi_{j+1}(\alpha)$.
\begin{eqnarray*}
(*)&=&
(p-1)\ol_i(\alpha)\left[ \pi_i(\alpha)\lambda_i(\alpha) p^{\alpha_i} x^\alpha
x_i +
\sum_{i<j<k}
\pi_j(\alpha)\lambda_j(\alpha)(p^{\alpha_j}-1)x^{\alpha+\epsilon^{i}-\epsilon^{j
}}x_j\right] + p\sigma_{ik}(x^\alpha)x_k \\
&=&
(p-1)\ol_i(\alpha)\left[ p^{\alpha_i} \pi_i(\alpha)  +
\sum_{i<j<k} \pi_j(\alpha)(p^{\alpha_j}-1)\right] x^{\alpha+\epsilon^i} +
p\sigma_{ik}(x^\alpha)x_k \\
&=&
(p-1)\ol_i(\alpha)\left[ \pi_{i+1}(\alpha)  +
\sum_{i<j<k} ( \pi_{j+1}(\alpha) - \pi_{j}(\alpha) ) \right]
x^{\alpha+\epsilon^i} +
p\pi_k(\alpha)\ol_i(\alpha)(p^{\alpha_k}-1)x^{\alpha+\epsilon^i} \\
&=&
(p-1)\ol_i(\alpha)\left[ \pi_{i+1}(\alpha)  +
\pi_{k}(\alpha) - \pi_{i+1}(\alpha) \right] x^{\alpha+\epsilon^i} +
p\pi_k(\alpha)\ol_i(\alpha)(p^{\alpha_k}-1)x^{\alpha+\epsilon^i} \\
&=&
\ol_i(\alpha)\left[ (p-1)\pi_{k}(\alpha)  +
p\pi_k(\alpha)(p^{\alpha_k}-1)\right]x^{\alpha+\epsilon^i} \\
&=& \ol_i(\alpha)(p^{\alpha_k+1}-1)\pi_{k}(\alpha)x^{\alpha+\epsilon^i}\\
&=&
\pi_{k}(\beta)\ol_i(\beta)\lambda_k(\beta)(p^{\beta_k}-1)x^{
\beta+\epsilon^i-\epsilon^k} = \sigma_{ik}(x^\beta),
\end{eqnarray*}
since $\lambda_k(\beta)=1=\lambda_k(\alpha)$ and
$\pi_k(\alpha)=\pi_k(\beta)$ as $\alpha$ and $\beta$ differ only
in the $k$th position.
\end{proof}

We will define a derivation $d:K_q[x_1, \ldots, x_n] \rightarrow
\Omega^1$ such that $d(x_i)=\omega_i$ for all $i$. For any
$\alpha\in\NN^n$ we set $d(x^\alpha)=\sum_{i=1}^n
\partial_i(x^\alpha)\omega_i$ where
\begin{equation}\partial_{i}(x^{\alpha})=\delta_i(\alpha)x^{\alpha-\epsilon^i}
\qquad \mathrm{and } \qquad
\delta_i(\alpha)=\pi_i(\alpha)\lambda_i(\alpha)\frac{p^{\alpha_{i}}-1}{p-1}.
\end{equation}
for all $i=1,\ldots, n$. Note that for $i,k$ we
have:
$$\delta_i(\alpha) = q^{\mp 1}\delta_i(\alpha \pm \epsilon^k),\:\:
\mbox{ if } i<k \qquad \mbox{and} \qquad \delta_i(\alpha) = p^{\mp
1}\delta_i(\alpha \pm \epsilon^k),\:\: \mbox{ if } i>k.$$

\begin{lemma}\label{manin-quantum-space-relations} The pair $(\partial, \sigma)$
is a right twisted multi-derivation of $K_q[x_1, \ldots, x_n]$
satisfying the equations $(\ref{qcomm})$ with respect to the multiplicatively
antisymmetric matrix $Q'$ whose
entries are $Q'_{ij}=p^{-1}q$ for $i<j$. In particular
\begin{equation}\label{qcomm_manin}
\partial_i\partial_j = pq^{-1} \partial_j\partial_i, \qquad \forall i<j
\end{equation}
holds as well as for all $i,k,j$:
\begin{eqnarray*}
\partial_i\sigma_{kj} &=& pq^{-1} \sigma_{kj}\partial_i, \qquad i < k \leq j \\
\partial_i\sigma_{kj} &=& pq^{-1} \partial_j\sigma_{ki}, \qquad k < i < j \\
\sigma_{ki}\partial_j &=& pq^{-1} \sigma_{kj}\partial_i, \qquad k < i < j \\
\partial_i\sigma_{ij}-pq^{-1}\partial_j\sigma_{ii} &=&
pq^{-1}\sigma_{ij}\partial_i-\sigma_{ii}\partial_j, \qquad i<j
\end{eqnarray*}
\end{lemma}

\begin{proof}
Let $\alpha, \beta \in \NN^n$. To prove that the
pair $(\partial, \sigma)$ is a right twisted multi-derivation, we
show the following $n$ equations hold
\begin{equation}\label{n-equations-multi-derivation}\partial_l(x^{\alpha}x^{
\beta})=\sum_{k}\partial_k(x^{\alpha})\sigma_{kl}(x^{\beta})+x^{\alpha}
\partial_l(x^{\beta}), \qquad l=1,...,n.\end{equation}
Since $x_ix_j=q^{-1}x_jx_i$ for $i>j$, we have
$x_i^{\alpha_i}x_j^{\beta_j}=q^{-\alpha_i\beta_j}x_j^{\beta_j}x_i^{\alpha_i}$
for $i>j$, and hence
$x^{\alpha}x^{\beta}=\mu(\alpha,\beta)x^{\alpha+\beta},$ where
$\mu(\alpha,\beta)=\prod_{1\leq r<s\leq n}q^{-\alpha_s\beta_r}.$
We then obtain
\begin{equation}\partial_l(x^{\alpha}x^{\beta})\nonumber
=\mu(\alpha,\beta)\delta_l(\alpha+\beta)x^{\alpha+\beta-\epsilon^{l}}
=\pi_l(\alpha+\beta)\lambda_l(\alpha+\beta)\frac{p^{\alpha_{l}+\beta_l}-1}{p-1}
\mu(\alpha,\beta)x^{\alpha+\beta-\epsilon^{l}}.
\end{equation}
On the other hand, we compute
\begin{eqnarray}\label{derivation-part1}
\nonumber \sum_{k=1}^{n}\partial_k(x^{\alpha})\sigma_{kl}(x^{\beta})
&=&\sum_{k=1}^{l-1}\partial_k(x^{\alpha})\sigma_{kl}(x^{\beta})+\partial_l(x^{
\alpha})\sigma_{ll}(x^{\beta})\\
\nonumber &=&\sum_{k=1}^{l-1}\delta_k(\alpha)
\pi_l(\beta)\ol_k(\beta)\lambda_l(\beta)
(p^{\beta_l}-1)x^{\alpha-\epsilon^{k}}x^{\beta+\epsilon^{k}-\epsilon^{l}}
+p^{\beta_l}\delta_l(\alpha) \pi_l(\beta) \ol_l(\beta)
\lambda_l(\beta)
x^{\alpha-\epsilon^{l}}x^{\beta}\\
 &=&\left[\pi_l(\beta)\frac{p^{\beta_l}-1}{p-1}\sum_{k=1}^{l-1}
\pi_k(\alpha) (p^{\alpha_k}-1)+ p^{\beta_l}
\pi_l(\alpha+\beta)\frac{p^{\alpha_l}-1}{p-1}\right]
\lambda_l(\alpha+\beta)\mu(\alpha,\beta)x^{\alpha+\beta-\epsilon^{l}}\\
\nonumber  &=&\left[\pi_l(\beta)\frac{p^{\beta_l}-1}{p-1}(\pi_l(\alpha)-1)+
p^{\beta_l}\pi_l(\alpha+\beta)\frac{p^{\alpha_l}-1}{p-1}\right]
\lambda_l(\alpha+\beta)\mu(\alpha,\beta)x^{\alpha+\beta-\epsilon^{l}}\\
\nonumber &=&\left[\pi_l(\alpha+\beta)\frac{p^{\alpha_l+\beta_l}-1}{p-1}
-\pi_l(\beta)\frac{p^{\beta_l}-1}{p-1}\right]\lambda_l(\alpha+\beta)\mu(\alpha,
\beta)x^{\alpha+\beta-\epsilon^{l}},
\end{eqnarray}
where the third equality holds because
$$\lambda_k(\alpha)\ol_k(\beta)x^{\alpha-\epsilon^{k}}x^{\beta+\epsilon^{k}
-\epsilon^{l}}
=\lambda_l(\alpha)\mu(\alpha,\beta)x^{\alpha+\beta-\epsilon^{l}}\qquad
\mbox{and} \qquad x^{\alpha-\epsilon^{l}}x^{\beta}=
\lambda_l(\beta)\mu(\alpha,\beta)x^{\alpha+\beta-\epsilon^{l}}.$$
The fourth equation follows since $\pi_k(\alpha)p^{\alpha_k}=\pi_{k+1}(\alpha)$.
As we also have
\begin{equation}\label{derivation-part2}
x^{\alpha}\partial_l(x^{\beta})=\delta_l(\beta)x^{\alpha}x^{\beta-\epsilon^{l}
}=\pi_l(\beta)\lambda_l(\alpha+\beta)\frac{p^{\beta_l}-1}{p-1}\mu(\alpha,
\beta)x^{\alpha+\beta-\epsilon^{l}}.
\end{equation}
We can conclude, combining (\ref{derivation-part1}) and
(\ref{derivation-part2}) that (\ref{n-equations-multi-derivation})
holds:
\begin{equation}\sum_{k=1}^{n}\partial_k(x^{\alpha})\sigma_{kl}(x^{\beta})+x^{
\alpha}\partial_l(x^{\beta})=\pi_l(\alpha+\beta)\lambda_l(\alpha+\beta)
\frac{p^{\alpha_{l}+\beta_l}-1}{p-1}\mu(\alpha,\beta)x^{\alpha+\beta-\epsilon^{l
}} = \partial_l(x^\alpha x^\beta).
\end{equation}

For any $i<j$ we have:
\begin{equation}\partial_i\partial_j(x^{\alpha})=
\delta_i(\alpha-\epsilon^j)\delta_j(\alpha)x^{\alpha-\epsilon^{i}-\epsilon^{j}}
= q^{-1}\delta_i(\alpha)p\delta_j(\alpha-\epsilon^i)
x^{\alpha-\epsilon^i-\epsilon^j} = pq^{-1}\partial_j\partial_i(x^\alpha)
\end{equation}

For $i<k<j,$ we have $\eta_{kj}(\alpha) = pq^{-1}
\eta_{kj}(\alpha-\epsilon^{i})$. Hence
\begin{equation}
\sigma_{kj}\partial_i(x^{\alpha}) = \delta_i(\alpha)
\eta_{kj}(\alpha-\epsilon^{i})
x^{\alpha-\epsilon^{i}+\epsilon^{k}-\epsilon^{j}} = p^{-1}q
\eta_{kj}(\alpha) \delta_i(\alpha)
x^{\alpha-\epsilon^{i}+\epsilon^{k}-\epsilon^{j}} = p^{-1}q
\partial_i(\sigma_{kj}(x^{\alpha}))
\end{equation}
which shows that  $\partial_i\sigma_{kj} = pq^{-1} \sigma_{kj}\partial_i$ for
all $i<k<j$.

For $i<k=j$, we have $\eta_{jj}(\alpha) = pq^{-1}\eta_{jj}(\alpha-\epsilon^i)$.
Thus
\begin{equation}\partial_i\sigma_{jj}(x^{\alpha})
= \eta_{jj}(\alpha)\delta_i(\alpha) x^{\alpha-\epsilon^{i}}
= pq^{-1}\delta_i(\alpha)\eta_{jj}(\alpha-\epsilon^i)x^{\alpha-\epsilon^{i}}
= pq^{-1}\sigma_{jj}(\partial_i(x^{\alpha})),
\end{equation}
showing $\partial_i\sigma_{jj}=pq^{-1}\sigma_{jj}\partial_i$ for $i<j$.

For $k<i<j$ using
$\eta_{kj}(\alpha)\delta_i(\alpha)=\eta_{ki}(\alpha)\delta_j(\alpha)$ we get:
\begin{eqnarray}
\partial_i\sigma_{kj}(x^{\alpha})
 &=&\eta_{kj}(\alpha)\delta_i(\alpha+\epsilon^{k}-\epsilon^{j})x^{
\alpha-\epsilon^{i}+\epsilon^{k}-\epsilon^{j}}\\
\nonumber  &=&pq^{-1}
\eta_{kj}(\alpha)\delta_i(\alpha)x^{\alpha-\epsilon^{i}+\epsilon^{k}-\epsilon^{j
}}\\
\nonumber  &=& pq^{-1}
\eta_{ki}(\alpha)\delta_j(\alpha)x^{\alpha-\epsilon^{i}+\epsilon^{k}-\epsilon^j}
\\
\nonumber  &=& pq^{-1}
\eta_{ki}(\alpha)\delta_j(\alpha+\epsilon^k-\epsilon^i)x^{\alpha-\epsilon^{i}
+\epsilon^{k}-\epsilon^j}
= pq^{-1}\partial_j\sigma_{ki}(x^\alpha)
\end{eqnarray}
showing
$\partial_i\sigma_{kj}(x^{\alpha})-pq^{-1}\partial_j\sigma_{ki}(x^{\alpha})=0.$
In a similar way, the relation
$$pq^{-1}\sigma_{kj}\partial_i(x^{\alpha})-\sigma_{ki}\partial_j(x^{\alpha})=0$$
holds for $k<i<j.$ Lastly, we show that the equations
$$\partial_i\sigma_{ij}(x^{\alpha})-pq^{-1}\partial_j\sigma_{ii}(x^{\alpha}) =
pq^{-1}\sigma_{ij}\partial_i(x^{\alpha})-\sigma_{ii}\partial_j(x^{\alpha}),
\qquad  i<j$$
are satisfied, because of the following equations for $i<j$
$$\sigma_{ii}\partial_j(x^{\alpha})=\frac{q^{-1}p^{\alpha_i}}{p-1}\eta_{ij}
(\alpha)
\pi_i(\alpha)\lambda_i(\alpha)x^{\alpha-\epsilon^{j}} = q^{-1}
\partial_j\sigma_{ii}(x^{\alpha})$$
$$\partial_i\sigma_{ij}(x^{\alpha})=\frac{q^{-1}}{p-1}\eta_{ij}(\alpha)
\pi_i(\alpha)\lambda_i(\alpha)(p^{\alpha_i+1}-1)x^{\alpha-\epsilon^{j}},$$
$$\sigma_{ij}\partial_i(x^{\alpha})=\frac{p^{-1}(p^{\alpha_i}-1)}{p-1}
\eta_{ij}(\alpha)
\pi_i(\alpha)\lambda_i(\alpha)x^{\alpha-\epsilon^{j}},$$

By using these equations we attain the equation:
$$\partial_i\sigma_{ij}(x^{\alpha})-pq^{-1}\partial_j\sigma_{ii}(x^{\alpha})
=-\frac{q^{-1}}{p-1}\eta_{ij}(\alpha)\pi_i(\alpha)\lambda_i(\alpha)x^{
\alpha-\epsilon^{j}}$$ and
$$pq^{-1}\sigma_{ij}\partial_i(x^{\alpha})-\sigma_{ii}\partial_j(x^{\alpha}
)=-\frac{q^{-1}}{p-1}\eta_{ij}(\alpha)\pi_i(\alpha)\lambda_i(\alpha)x^{
\alpha-\epsilon^{j}},$$ which completes the
proof the lemma.
\end{proof}

Denote by $\Omega=\bigwedge^{p^{-1}q}(\Omega^1)$ the quantum exterior algebra of
$\Omega^1$ over $K_q[x_1, \ldots, x_n]$ with respect to the matrix $Q'$.

\begin{theorem}  The derivation $d:K_q[x_1, \ldots, x_n]\rightarrow \Omega^1$
extends to a differential calculus $\bigwedge^{p^{-1}q}(\Omega^1)$ on $K_q[x_1,
\ldots, x_n]$. Furthermore the de Rham and the integral complex associated to
the
differential calculus $(\bigwedge^{p^{-1}q}(\Omega^1),d)$  are isomorphic.
\end{theorem}

\begin{proof}
The first statement follows from Proposition
\ref{ExtendingDiffStructToSkewExteriorAlgberas} and Lemma
\ref{manin-quantum-space-relations}. We have an upper-triangular
$\sigma=(\sigma_{ij})$ matrix by Lemma \ref{sigma-Manin-space}, of
which the diagonal entries $\sigma_{ii},i=1,\ldots ,n$ are
automorphisms. Hence we construct the corresponding
lower-triangular matrix $\bar\sigma$ according to
\cite[Proposition 3.3]{BrzezinskiKaoutitLomp}. The entries of
$\bsi$ are $\bsi_{ij}=0$ for $i<j$ and
$\bsi_{ii}=\sigma_{ii}^{-1}$ while
\begin{equation}\bsi_{ij}(x^{\alpha})=q \pi_i(\alpha)^{-1}
\ol_j(\alpha)^{-1}\lambda_i(\alpha)^{-1}(p^{-\alpha_i}-1)q^{\alpha_j-\alpha_i}x^{\alpha+\epsilon^{j}
-\epsilon^{i}},\end{equation} for $\alpha \in \NN^n$ and $i>j$.
Applying \cite[Proposition 3.3]{BrzezinskiKaoutitLomp} again
yields the map $\hsi$. The entries of $\hsi$ are $\hsi_{ij}=0$ for
$i>j$ and $\hsi_{ii}=\sigma_{ii}$ while
$\hsi_{ij}=p^{j-i}\sigma_{ij}$ for $i<j$.

By using these formulas for the entries of the matrices
$\bar\sigma(x^{\alpha})$ and $\hat\sigma(x^{\alpha})$, we
obtain an explicit expression for
\begin{eqnarray*}\partial_{i}^{\sigma}(x^{\alpha})&=&\sum_{1\leq j\leq k \leq
i}\bsi_{kj}
\circ \partial_{j}\circ\hsi_{ki}(x^{\alpha}).
\end{eqnarray*}
for any fixed $i=1,\ldots, n$. For $j< k < i$ we get:
\begin{equation*}
\bsi_{kj}\circ \partial_{j}\circ\hsi_{ki}(x^{\alpha})=-p^{i-k}
 \pi_j(\alpha)
\pi_k(\alpha)^{-1}(p-p^{-\alpha_k})(p^{\alpha_{j}}-1)\partial_i(x^\alpha)
 \end{equation*}
 while for $j=k<i$ we have:
\begin{equation*}
\bsi_{kk}\circ \partial_{k}\circ\hsi_{ki}(x^{\alpha}) =p^{i-k}
(p-p^{-\alpha_k}) \partial_i(x^\alpha)
\end{equation*}
Thus for any $k<i$ we get the partial sum:
\begin{eqnarray*}\Lambda_{k}&=&\sum_{j=1}^{k}\bar\sigma_{kj}\circ
\partial_{j}\circ\hat\sigma_{ki}(x^{\alpha})\\
&=& \sum_{j=1}^{k-1} -p^{i-k} \pi_j(\alpha) \pi_k(\alpha)^{-1}
(p-p^{-\alpha_k})(p^{\alpha_{j}}-1)\partial_i(x^\alpha)
+  p^{i-k}(p-p^{-\alpha_k})\partial_i(x^\alpha)\\
&=& \left[ 1 - \sum_{j=1}^{k-1} \pi_j(\alpha)(p^{\alpha_{j}}-1)
\pi_k(\alpha)^{-1}\right]
p^{i-k} (p-p^{-\alpha_k})\partial_i(x^\alpha)\\
&=& \left[ \pi_k(\alpha) - \pi_k(\alpha) + 1\right] \pi_k(\alpha)^{-1} p^{i-k}
(p-p^{-\alpha_k})\partial_i(x^\alpha)
= \pi_k(\alpha)^{-1} p^{i-k} (p-p^{-\alpha_k})\partial_i(x^\alpha)\\
\end{eqnarray*}

Similarly, for $k=i$ we have for $j<k=i$:
$\bsi_{ij}\circ \partial_{j}\circ\hsi_{ji}(x^{\alpha})
=-p \pi_j(\alpha) (p^{\alpha_{j}}-1)\pi_i(\alpha)^{-1}\partial_i(x^\alpha)$
and for $j=k=i$ we have $\bsi_{ii}\circ \partial_{i}\circ\hsi_{ii}(x^{\alpha}) =
p\partial_i(x^\alpha).$
This gives
$$\Lambda_{i}=\sum_{j=1}^{i}\bar\sigma_{ij}\circ
\partial_{j}\circ\hat\sigma_{ii}(x^{\alpha})=p\pi_i(\alpha)^{-1}
\partial_i(x^\alpha).$$ The sum of these partial sums $\Lambda_k$ yields:

\begin{eqnarray*}\partial_{i}^{\sigma}(x^{\alpha})
&=&\sum_{k=1}^{i}\Lambda_{k} =\sum_{k=1}^{i-1} \pi_k(\alpha)^{-1} p^{i-k}
(p-p^{-\alpha_k})\partial_i(x^\alpha) +
p\pi_i(\alpha)^{-1} \partial_i(x^\alpha) \\
&=& \left[\sum_{k=1}^{i-1} \pi_k(\alpha)^{-1} p^{i-k} (p-p^{-\alpha_k}) +
p\pi_i(\alpha)^{-1}\right] \partial_i(x^\alpha) \\
&=& p \lambda_i(\alpha) \frac{p^{\alpha_i}-1}{p-1} \left[ 1+ p^{i-1}
\sum_{k=1}^{i-1} p^{-k}(p^{\alpha_k+1}-1)
\left(\prod_{k<s<i} p^{\alpha_s}\right) \right] x^{\alpha-\epsilon^i}\\
&=& p \lambda_i(\alpha) \frac{p^{\alpha_i}-1}{p-1} \left[ 1+
p^{i-1} \sum_{k=1}^{i-1} \left( (p^{-(k-1)} \left(\prod_{k-1 <
s<i} p^{\alpha_s}\right) - p^{-k} \left(\prod_{k < s<i}
p^{\alpha_s}\right)\right) \right]
x^{\alpha-\epsilon^i}\\
&=& p \lambda_i(\alpha) \frac{p^{\alpha_i}-1}{p-1} \left[ 1+ p^{i-1} \left(
\pi_i(\alpha) - p^{-(i-1)} \right) \right] x^{\alpha-\epsilon^i}\\
&=&p^i\partial_i(x^\alpha)
\end{eqnarray*}

In order to apply Theorem \ref{uppertriangular}, we need to
calculate $\bdet\sigma$ as well as $\prod_{j} q'_{ij}$ where
$Q'=(q'_{ij})$ is the corresponding multiplicatively antisymmetric
matrix with $q'_{ij} = p^{-1}q$ for $i<j$. Let $\alpha \in \NN^n$.
By Theorem \ref{uppertriangular} it is enough to show that
$\partial_i^\sigma(x^\alpha) = \left(\prod_j
q'_{ij}\right)\bdet{\sigma}^{-1}(\partial_i(\bdet{\sigma}(x^\alpha)))$
holds, i.e.
$$ p^i\partial_i(x^\alpha) = \left(\prod_{j}
q'_{ij}\eta_{jj}(\alpha)\eta_{jj}(\alpha-\epsilon^i)^{-1}
\right)\partial_i(x^\alpha). $$ By the definition of $\eta_{ij}$
we obtain
$p^{-1}q\eta_{jj}(\alpha)\eta_{jj}(\alpha-\epsilon^i)^{-1} = 1$
for $i<j$ and
$pq^{-1}\eta_{jj}(\alpha)\eta_{jj}(\alpha-\epsilon^i)^{-1} = p$
for $i>j$, while
$\eta_{ii}(\alpha)\eta_{ii}(\alpha-\epsilon^i)^{-1} = p$. Hence
the product of the
$q'_{ij}\eta_{jj}(\alpha)\eta_{jj}(\alpha-\epsilon^i)^{-1}$ equals
$p^i$ and by Theorem \ref{uppertriangular} $K_q[x_1,\ldots, x_n]$
satisfies the strong Poincaré duality with respect to the
differential calculus $(\bigwedge^{p^{-1}q}(\Omega^1), d)$.

\end{proof}

\section{Conclusion}
Necessary and sufficient conditions to extend the associated FODC
$(\Omega^1,d)$ of a right twisted multi-derivation
$(\partial,\sigma)$ on an algebra $A$ to a full differential
calculus $(\Omega,d)$ on the quantum exterior algebra $\Omega$ of
$\Omega^1$  have been presented in this paper. A chain map between
the de Rham complex and the integral complex has been defined and
an criterion has been given to assure an isomorphism between the
de Rham and the integral complexes for free right upper-triangular
twisted multi-derivations whose associated FODC can be extended to
a full differential calculus on the quantum exterior algebra.
Easier criteria for FODCs with a diagonal bimodule structure have
been established and have been applied to show that a multivariate
quantum polynomial algebra satisfies the strong Poincaré duality in the sense of T.Brzezinski with respect to some canonical FODC. Lastly, we showed that for a
certain two-parameter $n$-dimensional (upper-triangular) calculus
over Manin's quantum $n$-space the de Rham and integral complexes
are isomorphic.

Future work will consist in extending our duality criteria to general FODCs having an upper-triangular bimodule structure.

\end{document}